\documentclass[12pt]{amsart}
\usepackage[utf8]{inputenc}
\usepackage{amsfonts, amsmath, amsthm}
\usepackage{paralist}
\usepackage{fullpage}

\newtheorem{Ex}  {Example}
\newtheorem{Def}[Ex]{Definition}
\newtheorem{Th}      [Ex]{Theorem} 

\newtheorem*{Th*}      {Theorem} 
\newtheorem{Le}      [Ex]{Lemma}

\begin{document}
\author{Alex Schreiber}
\address{Zentrum Mathematik - M3, Technische Universität München, 80290 München, Germany }
\email{schreiber@ma.tum.de}
\title [Probability depends monotonically on the bifix indicator.]{The probability of finding a fixed pattern in random data depends monotonically on the bifix indicator.}
\date{\today}
\maketitle
\begin{abstract}
We consider the problem of finding a fixed L-ary sequence in a stream of random L-ary data. It is known that the expected search time is a strictly increasing function of the lengths of the bifices of the pattern. In this paper we prove the related statement that the probability of finding the pattern in a finite random word is a strictly decreasing function of the lengths of the bifices of the pattern.
\end{abstract}

\section{Problem statement}

\begin{Def}[Bifix indicators]
Let $b = (b_1,b_2,\dots,b_n)$ be a word of length $n \geq 2$ of the alphabet $\{ 0,...,L-1 \}, L \geq 2$. 
We define the bifix indicator $h = (h_1,\dots,h_{n-1})$ of $b$ as a binary word of length $n-1$ such that
\[
h_i = \begin{cases}
 1 \text{ if }  (b_1, \dots, b_i) = (b_{n-i+1}, \dots, b_n)  \\
 0 \text{ otherwise. }
\end{cases}
\]
\end{Def}
Under $ h \leq h' $ we will understand $ h_i \leq h'_i $ for all $1 \leq i \leq n-1$ and under $h < h'$ that both $ h \leq h'$ and $h \neq h'$.

By speaking of a random word (of finite length) or a random sequence (of infinite length), here and elsewhere in this paper we mean a word or sequence with the members being selected independently with uniform probability $\frac1L$ of the alphabet.

\begin{Def}[Probabilities $P_k$ and $p_k$]
Fixing $L$, a word $b$ and the integer $k$, we let $P_k$ be the probability that the word $b$ appears (at least once) as substring in a random word of length $k$. We let $p_k$ be the probability that $b$ appears exactly once, namely  at the end in a random word of length $k$.
\end{Def}
One has $P_k = \sum_{i=1}^{k}p_i$ because if the searched substring appears at least once then there is an index $0 \leq i \leq k$ when it appears first.
Obviously, it holds $P_k=0$ and $p_k=0$ for $0 \leq k \leq n-1.$

One can interpret those probabilities also in the situation that one is looking for $b$ in an infinite sequence $(d_1, d_2, \dots)$ where $p_k$ is the probability for waiting time $k$, i.e. the probability that $b$ appears for the first time in $(d_1,\dots,d_k)$, i.e. $b=(d_{k-n+1}, \dots, d_k)$.

The aim of this paper is to proof the following theorem.

\begin{Th}
\label{mainTh}
Let $b$ and $b'$ be two words of length $n$ of the alphabet \mbox{$\{ 0,...,L-1 \}$.} 
Let $h_i$ and $h'_i$ be their bifix indicators and $P_k$ and $P'_k$ the corresponding probabilities as defined before.

We claim that
\begin{enumerate}[a)]
\item
if $ h = h' $ then $ P_k = P'_k $ for any $k \geq 0$ and that 
\item
if $ h < h' $ then $ P_k > P'_k $ for $k \geq k_0$
where 
\[k_0 := n + \min \{ 1 \leq i \leq n-1 | h_i=0 \text{ and } h'_i=1\}.\] 
For $0 \leq k < k_0$ it holds $P_k = P'_k$.
\end{enumerate}
\end{Th}
Note that part a) is easy and well-known. Our interest lies in part b).
In his PhD thesis \cite{hanus2010selected}, in 2010, Pavol Hanus states a special version of part b) of theorem \ref{mainTh}, namely for the case that $h = (0, \dots, 0)$, without giving a complete proof.
I am grateful to Alexander Mathis who pointed out this subject and reference to me.

Peter Nielsen gave in 
\cite{1055064}, in 1973,
a formula for the expectation of the waiting time until the first occurence of $b$ in a random word, the right-hand side of 
\[
\sum_{i=1}^{\infty}p_ii = L^n+\sum_{i=1}^{n-1} h_i L^i.
\]

This formula is actually a bit different from the one in \cite{1055064} because of deviations in our definitions. What we call $P_k$ would be $P_{k-n}$ for Nielsen. Thus he would also substract $n$ from the right-hand side.
Note that this formula quantizes the fact that the expected search time is an increasing function of the bifix pattern.
If one were just interested in this qualitative fact, this would follow independently from our paper as one has

\[
\sum_{i=1}^{\infty}p_ii = \sum_{k=0}^{\infty} (1-P_k).
\]

In order to prove our theorem it would be desirable to calculate an explicit formula for the value of $P_k$ like the one for the expected search time but we believe that there is none.
At least there is no affine dependence of $p_k$ (or, equivalently, $P_k$) on $h$ as there is for the expectation of the waiting time.
As an example, let $L=2$ and consider for $b$ the words 
\[
b^1=(1,0,0,0,0), b^2=(1,0,0,0,1),b^3=(1,0,0,1,0),b^4=(1,1,0,1,1).
\]
Then one has the corresponding bifix patterns
\[
h^1 = (0,0,0,0), h^2 = (1,0,0,0), h^3 = (0,1,0,0), h^4 = (1,1,0,0),
\]
consequently also (componentwise) $h^1+h^4=h^2+h^3,$ but in general $P_{k}^1+P_{k}^4 \neq P_{k}^2+P_{k}^3$, e.g. for $k=12$, as we calculated with the help of a computer.

However, recursive formulas are known for $p_k$ and $P_k$. In order to get explicit formulas, one would probably use those recursions and try to derive closed expressions from them. For that purpose, one would have to solve polynomial equations of degree $n$. As mentioned before, we do not believe it to be feasible to find explicit formulas.

Instead, we will show that the same recursions are obtained for certain Markov chains. Then we will compare the probabilities for the Markov chains, obtaining in that way a proof for the main theorem.

\section{Recursive formulas for the probabilities}

The probability $p_k$ amounts to the probability $\frac1{L^n}$ of ocurrence of the search pattern at position $k$ reduced by the probability of an additional ocurrence at an earlier time. Paying attention to the overlap structure given by the bifix indicator $h$ of the word $b$, one gets the formula 
\begin{align}
\label{eq:long-recursion}
p_k & = \frac1{L^n} - \frac1{L^n} \sum_{i=n}^{k-n} p_i - \sum_{i=1}^{n-1} h_i \frac1{L^{n-i}} p_{k-n+i} 
%\\
%&  = \frac1{L^n} - \frac1{L^n} p_n - \frac1{L^n} p_{n+1} - \dots - \frac1{L^n} p_{k-n} \\
%& - h_{1} L^{-n+1} p_{k-n+1} - \dots - h_{n-1} L^{-1} p_{k-1} \\
\text{ for any } k \geq n,
\end{align}
that appeared in \cite{DBLP:journals/corr/abs-cs-0508099}, in 2005. By knowing that $p_1 = \dots = p_{n-1} = 0$ the formula allows to calculate all the $p_k$. Among other things, this shows that the $p_k$ depend only on $h$, not on more information of $b$, as claimed by part a) of the theorem.  By taking the difference of the equation with the shifted equation where $k$ is replaced by $k+1$
one gets the following linear non-homogeneous $n+1$-term recursion for $p_k$ which was already given in \cite{383999}, in 1995:
\begin{align}
\label{eq:short-recursion}
p_{k+1} = p_k - \frac1{L^n} p_{k+1-n} - \sum_{i=1}^{n-1} h_i \frac1{L^{n-i}} (p_{k-n+i+1}-p_{k-n+i}) \text{ for any } k \geq n \\
p_1 = \dots = p_{n-1} = 0, p_n = \frac1{L^n}.
\end{align}

By summing up for $k = n, n+1, \dots, K$, but afterwards substituting $k$ for $K$, and noting that $p_n=\frac1{L^n}$ one gets

\begin{align}
\label{eq:Short-Recursion}
P_{k+1} = \frac1{L^n} + P_k - \frac1{L^n} P_{k+1-n} - \sum_{i=1}^{n-1} h_i \frac1{L^{n-i}} (P_{k-n+i+1}-P_{k-n+i}) 
\text{ for any } k \geq n \\ 
\label{eq:Short-Recursion2}
P_1 = \dots = P_{n-1} = 0, P_n = \frac1{L^n}.
\end{align}

\section{Markov chains for our problem}

In the following, we will turn our focus towards Markov chains. 
The motivation is that, given a fixed word $b$ and semi-infinite random data $d=(d_1, d_2, \dots)$ we are looking for the first appearence of $b$ in $d$, to that end observing $d$ ``from left to right''. 
Our first idea, that was not entirely conducive, was to associate to $b$ a Markov chain $\hat X(b) = (X_1,X_2,\dots)$ whose state $X_k$ would measure how good the chances are to encounter the subword $b$ in $(d_1, \dots, d_k)$ or some more letters of $d$. More specifically, $X_k = n$ would mean that $(d_1, \dots, d_k)$ contains $b$ and otherwise $X_k$ would be the greatest $i \; (0 \leq i \leq n-1)$ such that $(b_1, \dots b_i) = (d_{k-i+1}, \dots, d_k)$. The larger this $i$ the better would be the chances to find an instance of $b$ soon.

We want to prove a validity of a formula which uses only the bifix indicator of a search string $b$, not $b$ itself. However, there could be two different words $b^1$ and $b^2$ representing the same bifix class $h$ but, according to the construction described beforehand, having different associated Markov chains $\hat X(b^1)$ and $\hat X(b^2)$.

For that reason, we will associate in an improved ansatz a Markov chain $X(\bar h )$ to each bifix class $h$ such that, given a bifix class $h$ instead of a word $b$, one can often (but not always) find a representing word $b$ for $h$ such that $X( \bar h) = \hat X(b)$. 

This Markov chain will have the probability that $P_k = \Pr(X_k=n)$ whenever $P_k$ is defined as in the last section for a word $b$ and the Markov chain is the associated one to the bifix class $h$ of $b$.
 
\begin{Def}
We consider $n$-ary words $s= (s_0, \dots, s_{n-1})$ with 
\begin{equation}
\label{eq:jumpdown}
0 \leq s_i \leq i \text{ for } 0 \leq i \leq n-1.
\end{equation}
To such a word we associate the stationary Markov chain $X(s)=(X_0,X_1,\dots)$ where $X_0, X_1, \dots$ are random variables which map to the set $\{0,1,\dots,n\}$ with initial condition $X_0 \equiv 0$ (a.s.) and the transition probability $p_{ij} = \Pr(X_{k+1}=j | X_k=i)$ from state $i$ to state $j$ given by
\[
p_{ij} = \begin{cases}
1 & \text{ if } i=j=n \\
\frac1L & \text{ if } i \leq n-1, i+1=j \\
\frac{1}L & \text{ if } i \leq n-1, s_i = j>0 \\
\frac{L-2}L & \text{ if } i \leq n-1, s_i>j=0 \\
\frac{L-1}L & \text{ if } i \leq n-1, s_i=j=0 \\
0 & \text{ otherwise. }
\end{cases}
\]
Furthermore, we define
\[
P_k:=\Pr(X_k=n).
\]
\end{Def}

It will turn out soon that these numbers $P_k$ are related to those defined in the preceding section, therefore justifying our notation.

The initial condition and the transition probabilities determine the joint distribution of the Markov chain uniquely. The interpretation is that in each step with probability $\frac{L-2}L$ the transition is to state $0$, with probability $\frac1L$ the transition is from state $i$ to state $s_i$, and with probability $\frac1L$ the transition is from state $i$ to state $i+1$. Whenever the numbers $0$ and $s(i)$ are equal, the probabilities add up correspondingly.

To a given binary word $h= (h_1, \dots h_{n-1})$ we associate 
$s := \bar h := (0, 1-h_{n-1}, 1-h_{n-2}, \dots, 1-h_1)$.

\begin{Def}
Fix the search word $b$. For $k \geq 0$ and $ 0 \leq i \leq n$ we let
$P_k(i) := \Pr(X_{k+j}=n | X_j=i)$. 
Especially, one has $P_k(0) = P_k$ and $P_k(n) = 1$.
\end{Def}

\begin{Th}
\label{same}
Consider a binary word $h= (h_1, \dots h_{n-1})$.
The probabilities $P_k = P_k (\bar h, L)$ of the markov chain $X(\bar h)$ associated to the word $h$ are the same as the numbers $P_k = P_k(h,L)$ defined recursively in (\ref{eq:Short-Recursion}) and (\ref{eq:Short-Recursion2}) for $h$.
Thus, the probabilities $P_k = P_k(b,L)$ of finding a word $b$ with bifix pattern $h$ in a random word of length $k$ is the same as the probabilities $P_k = P_k (\bar h, L)$ for the corresponding markov chain $X(\bar h)$.
\end{Th}

\begin{proof}
We have
$P_k = \frac{L-1}{L} P_{k-1} + \frac1L P_{k-1}(1)$ for $k \geq 1$ or, equivalently, $P_k(1) = L P_{k+1}-(L-1) P_k$ for $k \geq 0$
and on the other hand
\[
P_k(1)= \frac1{L^{n-1}} + 
\sum_{i=1}^{n-1} \frac{L-2}{L^i}P_{k-i} + 
\sum_{i=1}^{n-1} \frac1{L^i} 
\begin{cases}
P_{k-i} \text{ if } h_{n-i}=1 \\
P_{k-i}(1) \text{ if } h_{n-i}=0
\end{cases}
\text{ for } k \geq n-1.
\]
One can see this equation as follows. With probability $\frac1{L^{n-1}}$, the state $1$ changes directly in $n-1$ steps to $n$. Otherwise the state increases by one exactly $i-1$ times in direct sequence and subsequently jumps to $0$ (first summation) or to $s_i=1-h_{n-i}$ (second summation).

Equating the right-hand sides of the last two equations, one gets
\begin{align*}
L P_{k+1}-(L-1) P_k
= & \frac1{L^{n-1}} + \sum_{i=1}^{n-1} \frac1{L^i}((L-2)P_{k-i} \\
& +\underbrace{P_{k-i}(1)}_{LP_{k-i+1}-(L-1)P_{k-i}}+h_{n-i}\underbrace{(P_{k-i}-P_{k-i}(1))}_{\underbrace{P_{k-i}-LP_{k-i+1}+(L-1)P_{k-i}}_{L(P_{k-i}-P_{k-i+1})})} \\
= & \frac1{L^{n-1}} + 
\sum_{i=1}^{n-1} \frac1{L^i}(-P_{k-i}+LP_{k-i+1} + h_{n-i}
L(P_{k-i} - P_{k-i+1} ) \\
= & \frac1{L^{n-1}} + \frac1L LP_k - \frac1{L^{n-1}}P_{k-n+1} + \sum_{i=1}^{n-1} \frac1{L^i} h_{n-i} L(P_{k-i} - P_{k-i+1} ),
\end{align*}
and, after dividing by $L$ and substituting $n-i$ for $i$,
\begin{align*}
P_{k+1} & = \frac1{L^n} + \frac{L-1}L P_k + \frac1L P_k - \frac1{L^n}P_{k-n+1} + \sum_{i=1}^{n-1} \frac1{L^i} h_{n-i} (P_{k-i} - P_{k-i+1} ) \\
 & = \frac1{L^n} + P_k - \frac1{L^n} P_{k+1-n} - \sum_{i=1}^{n-1} h_i \frac1{L^{n-i}} (P_{k-n+i+1}-P_{k-n+i}) \text{ for any } k \geq n,
% \\ & P_1 = \dots = P_{n-1} = 0, P_n = \frac1{L^n}.
\end{align*}
This is in fact the same as (\ref{eq:Short-Recursion}). Obviously, it also holds (\ref{eq:Short-Recursion2}).

\end{proof}

The reasoning towards the remainder of this section is chosen in a rather elementary way, avoiding any use of more than necessary probability theory. Another way would be to use the notions of stochastic domination and couplings between random variables that appears to arise quite naturally in the given situation. However, we did not take this route as it would actually not shorten the argument significantly.

\begin{Le}
\label{monoonk}
$P_k(i)$ is monotonely increasing in $k$, i.e.
\[
P_{k+1}(i) \geq P_k(i) \text{ for } k \geq 0, 0 \leq i \leq n. 
\]
\end{Le}
\begin{proof}
This is just because $X_k = n$ implies $X_{k+1} = n$. (Almost surely, but we will skip that specification in general.)
\end{proof}

\begin{Le}
\label{whenzero}
It holds $P_k(i) > 0$ if and only if $k+i \geq n$.
\end{Le}
\begin{proof}
 One has $P_k(i) \geq \frac1{L^{n-i}}$ for $k+i \geq n$. On the other hand, if $k+i < n$, then $P_k(i)=0$ because in general $X_{k+1} \leq X_k +1$, so if $X_0=i$ then $X_k \leq i+k < n$.

\end{proof}

\begin{Le}
\label{monooni}
$P_k(i)$ is monotonely increasing in $i$, i.e. 
\[
P_k(i+1) \geq P_k(i) \text{ for } k \geq 0, 0 \leq i \leq n-1.
\]
The inequality is strict if, in addition, $k+i+1 \geq n$, otherwise both sides equal $0$.
\end{Le}

\begin{proof}
We proceed by induction on $k$. $k=0$ is trivial because then only $i = n-1$ fulfills the requirement that $k+i+1 \geq n$, and then $P_0(n) = 1$ and $P_0(n-1) = 0$. For the induction step, we have 
\begin{align*}
P_k(i) & = \frac1L P_{k-1}(i+1) + \frac1L P_{k-1} (s_i) + \frac{L-2}L  P_{k-1}(0) \\
& \leq \frac1L P_{k-1}(i+1) + \frac{L-1}L P_{k-1}(i+1) = P_{k-1}(i+1) \leq P_k(i+1),
\end{align*}
where the first inequality is by induction hypothesis - note that $s_i \leq i+1$ and $0<i+1$ - and the last one is by lemma \ref{monoonk}.
In addition, if $k+i \geq n$, the first inequality is strict, by induction hypothesis. If instead $k+i+1 = n$, the last inequality is strict because then $P_{k-1}(i+1) = 0$ but $P_k(i+1) > 0$ by Lemma \ref{whenzero}.
\end{proof}

\begin{Th}(comparison of Markov chains)
\label{compareMarkov}
Suppose we are given two $n$-ary words $s = (s_0, \dots s_{n-1})$ and $s' = (s'_0, \dots s'_{n-1})$ which fullfill (\ref{eq:jumpdown}).
Assume further that $s > s'$, i.e. component-wise $s \geq s'$ but $s \neq s'$. Then for the associated Markov chains $X(s)$ and $X'(s) := X(s')$ we have $P_k > P'_k$ for all $k \geq k_0 := n + 1 + \min \{i-s_i \, | \, 0 \leq i \leq n-1 \text{ with } s_i>s_i' \}$ and $P_k = P'_k$ for all $k < k_0.$
\end{Th}

\begin{proof}
Fix $i^*$ with $i^*-s_{i^*} = \min \{i-s_i \, | \, 0 \leq i \leq n-1 \text{ with } s_i>s_i' \}$.

More generally than the theorem, we show that $P_k(i) \geq P'_k(i)$ and that

$P_k(i) > P'_k(i)$  if $k+i \geq n+1+i^*-s_{i^*}$ and $i \leq i^*$ and that

$P_k(i) = P'_k(i)$  if $k+i < n+1+i^*-s_{i^*}$. (For the remaining cases we make no statement about whether the inequality is strict.)

We proceed by induction on $k$. The case $k=0$ is trivial. For the induction step, we have 
\[P_k(i) = \frac1L P_{k-1}(i+1) + \frac1L P_{k-1} (s_i) + \frac{L-2}L  P_{k-1}(0) \]
and
\[P'_k(i) = \frac1L P'_{k-1}(i+1) + \frac1L P'_{k-1}(s_i') + \frac{L-2}L  P'_{k-1}(0). \]

Now we compare the summands on the right-hand side:
$P_{k-1}(i+1) \geq P'_{k-1}(i+1)$ and $P_{k-1}(0) \geq P'_{k-1}(0)$ by the induction hypothesis and 
$P_{k-1}(s_i) \geq P'_{k-1}(s_i')$ by the induction hyptothesis, $s_i \geq s_i'$ and Lemma \ref{monooni}. 

For the strictness statement if $k+i \geq n+1+i^*-s_{i^*}$ and $i \leq i^*$, note that $P_{k-1}(i+1) > P'_{k-1}(i+1)$ if $i < i^*$ by induction hypothesis, while if $i=i^*$ then $s_i > s_i'$ and so $P_{k-1}(s_i) \geq P'_{k-1}(s_i) >  P'_{k-1}(s_i')$ by the induction hypothesis and Lemma \ref{monooni} as $k-1+s_i \geq n+i^*-i = n$.

To show $P_k(i) = P'_k(i)$  if $k+i < n+1+i^*-s_{i^*}$, again we proceed by induction. 
For the induction step, note that the summands on the right-hand side of the recursion mutually agree. If $s_i= s'_i$ this is by induction hypothesis, otherwise one notes additionally that $k+i < n+1+i^*-s_{i^*} \leq n+1+i-s_i$, so $k -1+s_i < n$, so $P_{k-1} (s_i)=0$ by Lemma \ref{whenzero}.
\end{proof}

Now we are in a position to proof Theorem \ref{mainTh} concluding our work.

\begin{proof}[Proof of theorem \ref{mainTh}]

An argument for part a) of the theorem was already mentioned at the beginning of section 2.
Now for part b), we have the words $b$ and $b'$, their corresponding bifix patterns $h$ and $h'$ and define $s:=\bar h$ and $s' := \bar h'$. $s$ and $s'$ now fulfill the requirement of theorem \ref{compareMarkov}. The two theorems treat probabilities $P_k$ that are equal by theorem \ref{same} and numbers $k_0$ which are equal because of
\[
\min \{ 1 \leq i \leq n-1 | h_i=0 \text{ and } h'_i=1\} = 1+\min \{i-s_i \, | \, 0 \leq i \leq n-1 \text{ with } s_i>s_i' \}.
\]
Hence, we can use the implication of theorem \ref{compareMarkov} that $P_k > P'_k$ for all $k \geq k_0$ and $P_k = P'_k$ for all $k < k_0.$ which applies also for theorem \ref{mainTh} and we are done.

\end{proof}

\section{Summary of notation}

The following table is just thought as a guide for the reader. Our text does not completely follow it, especially for indices like $n$, $k$ and so on.

\begin{table}[!h]
\centering
\begin{tabular}{|c|c|} \hline
$L$ & size of the alphabet \\ \hline
$b$ & word that is searched for \\ \hline
$n$ & length of $b$ \\ \hline
$h$ & bifix indicator \\ \hline
$d$ & (finite) random word  or (infinite) random sequence \\ \hline
$k$ & length of a random word $d$ or \\  
    & index until which an infinite random sequence $d$ is observed \\ \hline
$P_k$ & probability of $b$ being a subword in a random word of length $k$ \\ \hline
$p_k$ & probability of $b$ being a subword in a random word of length $k$ \\ 
     & only at the end \\ \hline
$s$ & a word from $\{0,\dots,n-1\}^{\{0,\dots,n-1\}}$, will be derived from $h$ \\ \hline
$X(s)$ & a Markov chain derived from $s$ (and thus from $h$) \\ \hline
$p_{ij}$ & transition probability for a Markov chain form $i$ to $j$ \\ \hline
$P_k(i)$ & probability to get form $i$ to $n$ in at most $k$ steps  \\ \hline
$h', P'_k, X'_k$ & are derived from $b'$ the same way as $h, P_k, X_k$ from $b$.  \\ \hline

\end{tabular}
\end{table}

\bibliography{references}
\bibliographystyle{plain}

\end{document}